\documentclass[11pt]{amsart}

\usepackage{hyperref, color}
\usepackage{amssymb}

\theoremstyle{plain}
\newtheorem{theorem}                {Theorem}      [section]
\newtheorem*{theorem*}                {Theorem \ref{thm:appl}}

\newtheorem{corollary}    [theorem]  {Corollary}
\newtheorem{lemma}        [theorem]  {Lemma}

\theoremstyle{definition}

\numberwithin{equation}{section}

\begin{document}

\title
[Index bounds for closed minimal surfaces in 3-manifolds]
{Index bounds for closed minimal surfaces in 3-manifolds with the Killing property}

\author[Cavalcante]{Marcos Petr\'ucio Cavalcante}     
\address{Institute of Mathematics, Federal University of Alagoas, CEP 57072-970,
Macei\'o,  Alagoas,  Brazil}
\email{marcos@pos.mat.ufal.br}

\author[Oliveira]{Darlan F. de Oliveira}
\address{Universidade Estadual de Feira de Santana, Feira de Santana-BA, Brazil}
\email{darlanfdeoliveira@gmail.com}

\author[Silva]{Robson dos S. Silva}

\dedicatory{Dedicated to Professor Renato Tribuzy\\ on the occasion of his 75th birthday}

\date{\today}

\begin{abstract}   Let $\Sigma$ be a closed minimal surface
immersed in a Riemannian 3-manifold carrying an orthonormal Killing frame. 
This class of ambient spaces includes Lie groups with a bi-invariant metric. 
In this paper, we prove that the sum of the Morse index and the nullity of $\Sigma$  is bounded 
from below by a constant times its genus.
\end{abstract}

\maketitle

\section{Introduction} % Please enter the title of your first section (only the first letter of the title should be capital).

Minimal hypersurfaces are critical points for the area functional, while constant mean curvature (CMC)
hypersurfaces are critical points for the area functional for volume preserving variations.
If  such a hypersurface minimizes area up to the second order, then we say that it is \emph{stable}.

More generally, we may define the \emph{index} of a minimal or a CMC hypersurface as the dimension
of the maximal subspace in the respective space of variations, where the second variation 
of the area functional is negative definite. It is well known that the second variation of the area functional
is given by a quadratic form associated to the Jacobi operator 
$J$ 
(see Section \ref{pre} for details and definitions)
and the index of $\Sigma$ equals the number of negative eigenvalues of $J$, counting with
multiplicities.
Geometrically, the index indicates the number of distinct variations which decrease area.
In particular, a hypersurface is stable if and only if the index is zero. 
In this paper, let us assume that  all manifolds are  oriented.

It is well known that there exists a strong connection between  topology and  index of minimal and CMC hypersurfaces.
For instance, in \cite{SY}, Schoen and Yau proved that if $\Sigma^2$ is a 
closed  stable minimal surface embedded 
in a 3-manifold $(M^3,g)$ with nonnegative scalar curvature $R_g$, then either 
$\Sigma$ is a sphere or $\Sigma$ is a torus. 
In the Euclidean 3-space, if $\Sigma\subset \mathbb R^3$ is a complete minimal  stable surface,
 then $\Sigma$ is a flat plane. This result is a nice generalization of the
classical Bernstein  theorem and was proved independently by do Carmo and Peng \cite{dCP}, 
Fischer-Colbrie and Schoen \cite{FCS}, and Pogorelov \cite{pogo}. 
Very recently, Chodosh and Li, proved in \cite{CLi} this  is still true for complete stable 
minimal hypersurface in $\mathbb R^4$.
More generally, $\Sigma\subset \mathbb R^3$ 
has finite index if, and only if, $\Sigma$ is conformally equivalent  to a compact Riemann surface
with finitely many points removed, see \cite{FC} and \cite{Gu}. 
In fact, denoting by $Ind(\Sigma)$ the index of $\Sigma$, we have that  
\[
Ind(\Sigma)\geq \frac {1}{3}(2g+4k-5),
\]
where $g$ and $k$ denote the genus and the number of ends of $\Sigma$, respectively. 
This nice estimate was proved by Chodosh and Maximo \cite{CM, CM2}, and  other important previous
estimates were obtained by Ros in \cite{R06}, and by Grigor'yan, Netrusov and Yau  in \cite{GNY}.

For closed minimal hypersurfaces in the unit sphere $\mathbb S^{n+1}\subset \mathbb R^{n+2}$, 
Savo used  harmonic vector fields to prove in \cite{Savo} that the index is bounded from
below by a linear function of the first Betti number. This technique was  generalized by 
Ambrozio, Carlotto and Sharp in \cite{ACS}, for a class of ambient spaces
that admit a special embedding  in the Euclidean space. In both works, the authors make use of the
parallel canonical orthonormal frame in an appropriate  $\mathbb R^d$.
Related  results  were obtained by Mendes Radeschi \cite{MR2020},
Gorodski, Mendes and Radeschi \cite{GMR}, and
Chao Li \cite{chao}.
These results  strongly support Schoen and Marques-Neves conjecture \cite{M14, N14}, which claims that if the ambient
space $\ M^{n+1}$ is complete and has positive Ricci curvature, 
then there exists a positive constant $C$, depending only
on $M$, such that for any closed minimal hypersurface $\Sigma$ in $M$ it holds that
\begin{equation}\label{conj}
Ind(\Sigma)\geq C(b_1(\Sigma)+\ldots b_n(\Sigma)),
\end{equation}
where $b_i(\Sigma)$ denotes the $i$-th Betti number of $\Sigma$. This conjecture is still
opened, but an important contribution was recently made  by Song \cite{Song}.

%In \cite{ACS2}, Ambrozio, Carlotto and Sharp considered
%closed minimal hypersurfaces $\Sigma$ immersed in the flat torus $\mathbb T^{n+1}$. 
%In this case,   totally geodesic $n$-dimensional tori $\mathbb T^n$ inside $\mathbb T^{n+1}$ 
%are stable, but have first Betti number equals $n$, 
%and so we cannot expect an estimate like (\ref{conj}).
%However, they proved that
%\[
%Ind(\Sigma) \geq \frac{2}{n(n+1)} (b_1(\Sigma) - (2n-1)).
%\]
%If $n\geq 4$, it is assumed  there is a point  in $\Sigma$  where all principal curvatures are distinct.
%In particular, for closed oriented minimal surfaces immersed in $\mathbb T^3$, 
%they reobtain  the classical Ros' estimate proved in \cite{R06}.

For CMC hypersurfaces, a classical result due to 
Barbosa, do Carmo, and Eschenburg \cite{BdCE} , asserts that  geodesic spheres are
the only closed stable CMC hypersurfaces in simply connected space forms.
In particular, small geodesic spheres in $\mathbb T^3$ are examples of stable CMC surfaces.

The first result estimating the index of a CMC surface by 
the topology 
was obtained by  the first and the second authors  in \cite{CO}. 
It was proved that 
 the index of a closed CMC surface immersed
in $\mathbb R^3$ or in $\mathbb S^3$ is bounded from below by an explicit constant 
times the genus. This result was recently generalized in
\cite{AH} for the same class of ambient spaces considered in \cite{ACS}.

In this present paper, 
we consider closed minimal %and CMC 
surfaces immersed in
Riemannian $3$-manifolds with the \emph{Killing property}, that is,
supporting a global orthonormal frame of Killing vector fields.  This class of ambient space
was studied by D'Atri and Nickerson in \cite{DN} and  Tanno in \cite{T} and
contains all  Lie groups endowed with a bi-invariant metric. See
Section \ref{killing} for more details.

We also recall that the \emph{nullity} $Null(\Sigma)$ of a minimal hypersurface  
is the dimension of the space  of $L^2$ solutions of the Jacobi operator. In other words $Null(\Sigma)=\dim\{u\in W^{1,2}: Ju=0\}$.
In \cite{chao}, Chao Li proved that the nullity is also related with the topology of the hypersurface.

In our first result we use a novel approach to prove that the sum of the index and the nullity of
minimal surfaces immersed in this class of spaces is
bounded from by a multiple of its genus.
Since index and nullity are non negative integers, we can use the ceiling function
$
\lceil x\rceil =\min \,\{n\in \mathbb {Z} \mid n\geq x\}$ in our estimates.

%Our first result read as 
\begin{theorem}\label{t1}
If $\Sigma^2$ is a closed minimal surface immersed in a Riemannian $3$-manifold $M^3$ with
the Killing property, then
\[
Ind(\Sigma)+Null(\Sigma)\geq
\left\lceil\frac{g(\Sigma)}{3}\right\rceil.
\]
Moreover, if $M$ has positive Ricci curvature, then
\[
Ind(\Sigma)\geq \left\lceil\frac{g(\Sigma)}{3}\right\rceil.
\]
\end{theorem}

As a consequence of the proof we have
\begin{corollary}\label{c1}
If $\Sigma^2$ is a closed CMC surface immersed in a Riemannian $3$-manifold $M^3$ with
the Killing property, then
\[
Ind(\Sigma)\geq \left\lceil\frac{g(\Sigma)}{3}-1\right\rceil.
\]
\end{corollary}

We point out that our bounds
can be found already in the literature,
sometimes better ones, 
when applied  to  the list of known examples of 
$3$-manifolds with the Killing property.

\section{Preliminaries}\label{pre} % Please enter the title of your second section.
In this section we present some preliminaries definitions and notations we will use to prove  our results. 
\subsection{The Killing property}\label{killing}
We say that a  Riemannian manifold $M$ has the \emph{Killing property} if, in some neighborhood
of each point of $M$, there exists an orthonormal Killing frame, that is, an orthonormal frame $\{X_1, \ldots, X_n\}$ 
such that each  $X_i$ is a Killing vector field, in the sense that it generates infinitesimal isometries. 

This concept was introduced in \cite{DN} by D'Atry and Nickerson in 1968, and among other things they proved
that Riemannian manifolds with Killing property have nonnegative sectional curvature.  
Next, Tanno in \cite{T} proved that in the 3-dimensional case, Killing property implies that $M$ has
constant (nonnegative) sectional curvature.

In this paper we will consider Riemannian manifolds endowed with a \emph{global} orthonormal 
Killing frame  as ambient spaces for minimal hypersurfaces.

The most simple example is the Euclidean space. More generally, 
we can  check that Lie groups endowed with  a bi-invariant metric have a global 
orthonormal frame of Killing vector fields. In fact, fixed an orthonormal basis in the Lie algebra, the 
corresponding left (or equivalently right) invariant frame is an orthonormal Killing frame. 
In the 3-dimensional case such spaces were classified in \cite{MR2002821}. They are
$\mathbb R^3$, $\mathbb S^3$, $\mathbb T^3=\mathbb S^1\times \mathbb S^1\times \mathbb S^1$,
$\mathbb S^1\times \mathbb S^1\times \mathbb R$,
$\mathbb S^1\times  \mathbb R^2$, and the real projective space $\mathbb R\mathbb P^3$ (=$SO(3)$).

It is an interesting fact that, among the canonical spheres, only $\mathbb S^1$, $\mathbb S^3$ 
and $\mathbb S^7$ have a global orthonormal Killing frame (see \cite{DN}). 
Finally, it is easy to see that the Riemannian product of  manifolds with the Killing property
also have the Killing property.

\subsection{Stability of closed hypersurfaces}\label{stab}

Let $M^{n+1}$ be a Riemannian manifold, and let $\textrm{Ric}$ denote its Ricci tensor.
If $\Sigma^n$ is a closed, two-sided, minimal hypersurface in $M$,
we denote by $N$ a globally defined unit normal vector along $\Sigma$, and by $A$ its second fundamental form. 
For simplicity of notation, we also use $A$ to denote the shape operator of 
$\Sigma$.
 In this case, the second variation of the area
functional of $\Sigma$ is given by the  quadratic form
\[
Q(u,u)= -\int_\Sigma uJu\,dv, %\quad \forall u\in C^\infty(\Sigma),
\]
where $J = \Delta + \|A\|^2+\textrm{Ric}(N)$ is the Jacobi operator acting in the Sobolev space
$W^{1,2}(\Sigma)$.  We can easily  see that $J$ is a self-adjoint elliptic operator and its spectrum
is given by a sequence of eigenvalues
 diverging to $+\infty$, say 
\[
\lambda_1<\lambda_2\leq \ldots \lambda_k \leq \ldots
\]

In particular, if $\{\phi_1, \phi_2,\ldots, \phi_k, \ldots\}$ is an orthonormal basis of eigenfunctions of $J$,
and if we denote by $\mathcal S_k=\langle \phi_1, \ldots, \phi_k\rangle^\perp$ the subspace of
$W^{1,2}(\Sigma)$ orthogonal to  the first $k$ eigenfunctions of $J$, then the min-max 
characterization of eigenvalues  implies that
\[
\lambda_k  = \inf_{u\in\mathcal S_{k-1}}\frac{\int_\Sigma uJu\,dv}{\int_\Sigma u^2\,dv}.
\]

The  \emph{index} $Ind(\Sigma)$ of $\Sigma$  is defined as the number of negative eigenvalues 
of $J$ in $W^{1,2}(\Sigma)$. Equivalently, the index of $\Sigma$ is
the maximal dimension of a subspace of $W^{1,2}(\Sigma)$
where $Q$ is negative defined.
The \emph{nullity}  of $\Sigma$, denoted by $Nul(\Sigma)$, is the  dimension of the subspace of eigenfunctions 
corresponding the eigenvalue $\lambda = 0$.

If $\Sigma$ is a CMC hypersurface, that is, $H\neq 0$ and constant, the index and the nullity are defined in the same way, but considering the space of test functions as the  subspace of $W^{1,2}(\Sigma)$
orthogonal to constants. In fact, the test functions may
have zero mean integral in order to generate volume
preserving variations.

%If $\Sigma$ is a CMC hypersurface, 
%the second variation is  also given by the  quadratic form $Q$ associated to the the Jacobi operator, 
%but now restricted to the space $\mathcal F = \{u\in W^{1,2}(\Sigma): \int_\Sigma u\, dv=0 \}$.
%
%If we denote by $\mu_1<\mu_2\leq \ldots \mu_k \leq \ldots$, the eigenvalues of $J$ restricted to
%$\mathcal F$, then the index of $\Sigma$ is the number of negative $\mu_k$, and it will be
%denoted by $Ind_w(\Sigma)$. From the work of Barbosa and Bérard \cite{BB}, we know that 
%\[
%Ind_w\leq Ind+1.
%\]

%\begin{theorem} \label{theorem1}
%oi
%\end{theorem}
%
%\begin{proof}
%Here is the proof of the theorem.
%\end{proof}
%
%\begin{corollary} \label{corollary1}
%       This is a corollary of Theorem \ref{theorem1}.
%\end{corollary}
%\begin{proof}
%Here is the proof of the corollary.
%\end{proof}
%
%\begin{remark}
%		We remark that Definition \ref{definition1} is correct.
%\end{remark}

\subsection{Test functions and harmonic vector fields}\label{test}

%Let $M^n$ be a compact Riemannian manifold. 
We denote by $\Omega^p(\Sigma)$ the space of
$p$-forms on $\Sigma$ and by $\mathcal H^p(\Sigma)$ the subspace of harmonic $p$-forms,
that is, 
\[
\mathcal H^p(\Sigma)=\{\omega \in \Omega^p(\Sigma): \Delta \omega = d\delta \omega+ \delta d \omega=0\}.
\]
%those forms $\omega$ such that the Hodge-Laplacian, $\Delta \omega = d\delta \omega+ \delta d \omega$
%vanishes.

We recall from Hodge-de Rham theorem that $\mathcal H^p(\Sigma)$ is 
isomorphic to the $p$-th de Rham cohomology
group $H^p(\Sigma)$, and so $\dim \mathcal H^p(\Sigma)=b_p(\Sigma)$ is 
the $p$-th Betti number of $\Sigma$. In particular, if $\Sigma$ is 2-dimensional, then 
$b_1(\Sigma) =2g(\Sigma)$, where $g(\Sigma)$ denotes
the topological genus of $\Sigma$.

The Riemannian metric of $\Sigma$ induces the so called \emph{musical
isomorphism} between $1$-forms $\omega\in \Omega^1(\Sigma)$ 
and vector fields $\xi\in T\Sigma$ given by the following identity
\[
\omega(X) = \langle \xi, X\rangle,
\]
for all $X\in T\Sigma$. In this case, we say that $\omega$ and $\xi$ are dual, and we write 
$\xi = \omega^\sharp$ or $\omega = \xi^\flat$. 
For instance, using this isomorphism we have that 
\begin{equation}\label{div}
\delta \omega = \textrm{div}\, \xi = - tr \nabla \xi.
\end{equation}
Moreover, we define the \emph{Hodge Laplacian} of $\xi$ as the vector field dual to
the Hodge Laplacian of $\xi^\flat$, that is, $\Delta \xi = (\Delta \xi^\flat)^\sharp.$
In particular, we say that
$\xi$ is a \emph{harmonic vector field}  if and only if $\omega$ is a harmonic form.

On the other hand, for vector fields on $\Sigma$ we also have the so called \emph{rough Laplacian}, which is defined as
\[
\nabla ^\star \nabla \xi = -\sum_{k=1}^n (\nabla_{e_k}\nabla_{e_k}\xi-\nabla_{\nabla_{e_k}e_k}\xi),
\]
where $\{e_1,\ldots, e_n\}$ is a local orthonormal frame on $\Sigma$. These two Laplacians are related
by the Bochner-Weitzenb\"ock formula
\begin{eqnarray}\label{bochner}
\Delta \xi = \nabla ^\star \nabla \xi + \textrm{Ric}_\Sigma(\xi).
\end{eqnarray}

Now, let us assume that $M^{n+1}$ has the Killing property, and let denote by 
$\{X_1,\ldots, X_{n+1}\}$  an orthonormal
frame of Killing vectors. If $\Sigma\subset M$ is a closed minimal hypersurface, we denote by 
$E_i$ the orthogonal projection of $X_i$ onto the tangent space of $\Sigma$, that is,
\[
E_i = X_i-g_iN,
\]
where $g_i=\langle X_i, N\rangle $. Following ideas in \cite{Savo} and \cite{ACS}, given a harmonic
vector field $\xi = \omega^\sharp$ on $\Sigma$  we will consider
\[
w_i=\langle \xi, E_i\rangle, %\quad\textrm{and}\quad u_{ij}=(N^\flat\wedge\omega)(X_i, X_j),
\]
as test functions to estimate index and nullity of $\Sigma$. 
%As we will see, the functions $w_i$ will be used in the 2-dimension case, 
%while the functions $u_{ij}$ will be used in the higher dimension case when the ambient space is the flat torus. 

\section{Minimal surfaces in 3-manifolds with Killing property}

In this section, $M^3$  denotes a 3-dimensional Riemannian manifold 
endowed with a global orthonormal
frame of Killing vectors, $\{X_1, X_2,X_3\}$. 
In the next lemma we use the notations presented in
Section \ref{pre}.

\begin{lemma}\label{l1}
Let $\Sigma$  be a closed oriented CMC surface immersed in $M^3$. 
Then we have the following assertions:
\begin{enumerate}
\item $\langle \nabla_XE_i, Y\rangle +\langle \nabla_YE_i, X\rangle=2g_i\langle AX, Y\rangle$,
for all $X,Y\in T\Sigma$.
\item $\langle \nabla E_i, \nabla \xi\rangle =g_i\langle A, \nabla \xi\rangle $.
\item \emph{div}$E_i=-2Hg_i$.
\end{enumerate}
\end{lemma}
\begin{proof}
The first assertion follows taking the tangent part of the Killing equation for the vector fields $X_i$.

To prove the second assertion, we consider $\nabla E_i$ and $\nabla \xi$ as 2-forms using
the musical isomorphism.
So, $\nabla E_i(X,Y) = \langle \nabla_X E_i, Y\rangle$.  It allow us to consider the 
algebraic  dual of $\nabla E_i$, defined by the following equation 
\[
(\nabla E_i)^t (X,Y)= \langle X, \nabla_Y E_i\rangle.
\]
That is, $(\nabla E_i)^t(X,Y) = \nabla E_i(Y,X)$, and 
from Assertion 1, we have, $\nabla E_i + (\nabla E_i)^t= 2g_iA$. 

On the other hand,
since $\xi$ is harmonic, $\nabla \xi$ is  a symmetric tensor. Thus, if 
$\{e_1,e_2\}$ is a local orthonormal frame, geodesic at $p\in\Sigma$ we have
\begin{eqnarray*}
\langle(\nabla E_i)^{t},\nabla\xi\rangle &=&\sum_{k,j}(\nabla E_i)^{t}(e_k,e_j)\cdot\nabla\xi(e_k,e_j) \\
&=&\sum_{k,j}(\nabla E_i)(e_j,e_k)\cdot\nabla\xi(e_j,e_k) \\
&=&\sum_{k}\sum_{j}\langle\nabla_{e_k}E_i,e_j\rangle\langle\nabla_{e_k}\xi,e_j\rangle \\
&=&\sum\langle\nabla_{e_k}E_i,\nabla_{e_k}\xi\rangle \\
&=&\langle\nabla E_i,\nabla\xi\rangle.
\end{eqnarray*} 

So,
\begin{eqnarray*}
\langle\nabla E_i,\nabla\xi\rangle &=&\frac{1}{2}\langle\nabla E_i+(\nabla E_i)^{t},\nabla\xi\rangle+
\frac{1}{2}\langle\nabla E_i-(\nabla E_i)^{t},\nabla\xi\rangle \\
&=&\frac{1}{2}\langle\nabla E_i+(\nabla E_i)^{t},\nabla\xi\rangle \\
&=&\frac{1}{2}\langle 2g_iA,\nabla\xi\rangle \\
&=&g_i\langle A,\nabla\xi\rangle.
\end{eqnarray*}

Finally, denoting by $h$ the induced metric on $\Sigma$, we have
\begin{eqnarray*}
\text{div}E_i &=&-\text{tr}(\nabla E_i) \\
&=&-\langle \nabla E_i ,h\rangle \\
&=&-\frac{1}{2}\langle\nabla E_{i}+(\nabla E_{i})^{t},h\rangle \\
&=&-\frac{1}{2}\langle 2g_i A,h\rangle \\
&=&-g_i\langle A,h\rangle \\
&=&-g_i\text{tr}(A) \\
&=&-2Hg_i.
\end{eqnarray*}
This concludes the lemma.
\end{proof}

In the next lemma we use  auxiliary test functions associated to the Hodge star of $\omega$ in $\Sigma$. More precisely, we define
\[
\bar w_i = \star \omega(E_i) = \langle \star \xi, E_i\rangle.
\]

Moreover, we point out that due to our convention in (\ref{div}) the Green's formula  for 
closed manifolds reads as
\[
\int_\Sigma w\,\textrm{div} X\,dv = \int_\Sigma \langle \nabla w, X\rangle dv,			
\]
for $w\in C^1(\Sigma )$ and $X\in T\Sigma$.
\begin{lemma}\label{l2}
\[
\sum_i \big( \int_\Sigma w_iJw_i dv+ \int_\Sigma \bar w_iJ\bar w_i dv\big)= 
 -2\int_\Sigma \emph{Ric}(N,N)|\xi|^2 - 4H^2\int_\Sigma |\xi|^2dv.
\]
\end{lemma}
\begin{proof}
In this proof we will use Einstein sum notation. Recall that
 $\{e_1,e_2\}$ denotes a local orthonormal frame, geodesic at $p\in\Sigma$.
 
From the first assumption of Lemma \ref{l1} we have
\begin{eqnarray*}
\nabla w_i &=& % \nabla\langle \xi, E_i\rangle
\nabla_{E_i} \xi+2g_i\langle  A\xi, e_k\rangle e_k - \langle  e_k, \nabla _\xi E_i\rangle e_k\\
%&=& \nabla_{E_i} \xi - \nabla_\xi{E_i} 
&=& [E_i, \xi]+2g_i   A\xi .
\end{eqnarray*}

On the other hand, since $\textrm{div}\, \xi = 0$, computing the divergence of the 2-form 
$\xi^\flat \wedge E_i^\flat$ we obtain
\begin{eqnarray*}
\delta (\xi^\flat \wedge E_i^\flat) &=&  \textrm{div} (\xi) E_i -\textrm{div} (E_i) \xi - [\xi, E_i]\\
&=& 2Hg_i\xi + [E_i,\xi].
\end{eqnarray*}

That is,
\[
\nabla w_i = \delta (\xi^\flat \wedge E_i^\flat)  - 2Hg_i \xi+2g_i   A\xi .
\]

Now we compute the energy of $w_i$ as follows
\begin{eqnarray*}
\int_\Sigma |\nabla w_i |^2\, dv
&=& \int_\Sigma  \langle \nabla w_i,  \delta ((\xi^\flat \wedge E_i^\flat)  - 
2Hg_i\xi +2g_i   A\xi\rangle  dv \\
&=&- 2H \int_\Sigma \langle \nabla w_i,   g_i\xi\rangle dv 
+2 \int_\Sigma \langle \nabla w_i,  g_i   A\xi\rangle dv \\
&=&- 2H \int_\Sigma   w_i\,  \textrm{div} (g_i\xi) dv 
+2 \int_\Sigma   w_i \, \textrm{div}  (g_i A\xi) dv \\
&=&2H \int_\Sigma   w_i  \langle \nabla g_i,\xi\rangle dv 
-2 \int_\Sigma   w_i   \langle \nabla g_i,A\xi\rangle dv, 
\end{eqnarray*}
since $\textrm{div}\, A\xi$ also vanishes.
Similarly, we have
\[
\int_\Sigma |\nabla \bar w_i |^2\, dv = 
2H \int_\Sigma   \bar w_i  \langle \nabla g_i, \star \xi\rangle dv 
-2 \int_\Sigma  \bar w_i   \langle \nabla g_i,A\star\xi\rangle dv.
\]
Thus,
\[
\int_\Sigma |\nabla w_i |^2\, dv+ \int_\Sigma |\nabla \bar w_i |^2\, dv
= 2H \int_\Sigma    \langle  E_i \nabla g_i \rangle |\xi|^2 dv 
-2 \int_\Sigma   \langle E_i, A\nabla g_i\rangle |\xi|^2dv.
\]

Now let us compute these terms separately.  
Using the harmonicity of $\xi$  we have
\begin{eqnarray*}
\int_\Sigma   \langle E_i, A\nabla g_i\rangle |\xi|^2dv 
&=&  \int_\Sigma   g_i \,\textrm{div} (|\xi|^2 AE_i)dv\\
&=& \int_\Sigma   g_i |\xi|^2\,\textrm{div} ( AE_i)dv		\\
&=& - \int_\Sigma   g_i |\xi|^2\langle \nabla_{e_k} (AE_i),e_k\rangle dv\\
&=& - \int_\Sigma   g_i |\xi|^2\langle (\nabla_{e_k} A)E_i+ A(\nabla_{e_k}E_i),e_k\rangle dv\\
&=& - \int_\Sigma   g_i |\xi|^2\langle E_i,  (\nabla_{e_k} A)e_k\rangle
-  \int_\Sigma   g_i |\xi|^2\langle A, \nabla E_i \rangle dv\\
&=&-  \int_\Sigma   g_i |\xi|^2\langle A, \nabla E_i \rangle dv\\
&=&-  \frac 1 2 \int_\Sigma   g_i |\xi|^2\langle A, \nabla E_i +(\nabla E_i )^t\rangle dv\\
&=&-  \frac 1 2 \int_\Sigma   g_i |\xi|^2\langle A, 2g_iA \rangle dv\\ 
&=& - \int_\Sigma   g_i^2 |\xi|^2\| A\|^2 dv.
\end{eqnarray*}
%since $\sum g_i ^2 =1.$

Finally, using Assertion 3 of Lemma \ref{l1}, we have
\begin{eqnarray*}
H \int_\Sigma    \langle  E_i \nabla g_i \rangle |\xi|^2 dv 
&=&  H \int_\Sigma    g_i \,\textrm{div}( |\xi|^2  E_i)dv\\
&=& H \int_\Sigma    g_i \,|\xi|^2\textrm{div}( E_i)dv\\
&=& -2H^2 \int_\Sigma    g_i^2 \,|\xi|^2dv.
\end{eqnarray*}

Therefore, we get
\[
\sum_i \big( \int_\Sigma |\nabla w_i |^2\, dv+ \int_\Sigma |\nabla \bar w_i |^2\, dv\big)
=2\int_\Sigma \|A\|^2|\xi|^2dv-4H^2\int_\Sigma |\xi|^2 dv,
\]
and the lemma follows.

%%%%%%%%%%%%%%%%%%%%%%%%%%%%%%%%%
\subsection{Proof of Theorem \ref{t1}}
%Now we are in position to prove Theorem
Let us recall the notation fixed in Section \ref{stab}. So, let
$ \lambda_1<\lambda_2\leq \ldots \lambda_k \leq \ldots$ be the sequence of eigenvalues
of the Jacobi operator $J$, and let  $\{\phi_1, \phi_2,\ldots, \phi_k, \ldots\}$  be
an orthonormal basis of eigenfunctions of $J$. We also
set $\mathcal S_k=\langle \phi_1, \ldots, \phi_k\rangle^\perp$.

Let us consider a harmonic vector field $\xi\in \mathcal H ^1(\Sigma)$ 
such that the test functions 
$w_i$ and $\bar w_i$ belong in the space $\mathcal S_{k-1}$, for some $k\in \mathbb N$
and all $i\in\{1,2,3\}$. This is equivalent to find $\xi\in \mathcal H ^1(\Sigma)$  such that
the system of $6(k-1) $ homogeneous linear equations
\begin{eqnarray}
\int_{M}w_{i}\phi_{j}dv=\int_{M}\bar{w}_{i}\phi_{j}dv=0, \quad i\in\{1,2,3\}, \, j\in\{1, \ldots ,k-1\}.
\end{eqnarray}

Since  $\dim \mathcal H^1(\Sigma)=b_1(\Sigma)=2g(\Sigma)$,  we have a non trivial solution if 
$2g(\Sigma)>6(k-1)$. In this case, we have from Lemma \ref{l2} 
and  the min-max characterization that 
\[
\lambda_k \leq 0,
\]
that is $Ind(\Sigma)+ Nul(\Sigma)\geq k$. On the other hand, since $2g(\Sigma)>6(k-1)$,
have that $k\geq g(\Sigma)/3$, and we have done.

\subsection{Proof of Corollary \ref{c1}}

This corollary follows immediately from the general result proved by Barbosa and B\'erad in \cite{BB}, but it also
follows from our proof above. In fact, 
we just need to guarantee that the functions
$w_i$ and $\bar w_i$ satisfy the zero mean integral condition. In other words,  it
means that additionally  they solve $6$ new equations, namely
\begin{eqnarray*}
\int_{M}w_{i}dv=\int_{M}\bar{w}_{i}dv=0, \quad i\in\{1,2,3\}.
\end{eqnarray*}
So we have a non trivial solution if $2g(\Sigma)>6k$ and the result follows.

\end{proof}

\section*{Acknowledgements}
The authors wish to thank Celso Viana for his interesting comments 
and helpful discussion during the preparation of this article and
the referee for valuation suggestions.
The first author was partially supported by 
Brazilian National Council for Scientific and Technological Development (CNPq Grants 309733/2019-7, 201322/2020-0 and 405468/2021-0) and 
Coordena\c c\~ao de Aperfei\c coamento de Pessoal de N\'ivel Superior - Brasil 
(CAPES-Cofecub 88887.143161/ 2017-0 and
CAPES-MathAmSud 88887.368700/2019-00).

% Please format your references as follows in your main tex file.
% Using BibTex is not recommended but can be handled.
% The list of references should follow alphabetical order of the authors' last (family) name. If there are more papers by the same author(s), these should be in chronological order.
% Abbreviations of names of journals should follow the Mathematical Reviews (see http://www.ams.org/msnhtml/serials.pdf). Page numbers should be written with an en dash, e.g. 15--23.

% Please make sure that you list only those entries that are cited in your paper.

\bibliographystyle{amsplain}
\bibliography{bibliography}

\end{document}